%% file: 1main.tex
\documentclass[a4paper,10pt]{amsart}
\usepackage{amsmath,amssymb,latexsym, mathtools}
\usepackage{amscd}
\usepackage{amsmath}
\usepackage{amssymb}
\usepackage{amsthm}
\usepackage{tikz-cd}
\usepackage[frame,cmtip,curve,arrow,matrix,line,graph]{xy}
\usepackage{comment}

\theoremstyle{plain}
\newtheorem{thm}{Theorem}[section]
\newtheorem{prop}[thm]{Proposition}
\newtheorem{lem}[thm]{Lemma}

\newtheorem{defi}[thm]{Definition}
\newtheorem{rmk}[thm]{Remark}

\newtheorem{prop-defi}[thm]{Proposition-Definition}
\newtheorem{thm-defi}[thm]{Theorem-Definition}
\newtheorem{lem-defi}[thm]{Lemma-Definition}

\newtheorem{conj}[thm]{Conjecture}

\newcommand{\aA}{\mathcal{A}}
\newcommand{\bB}{\mathcal{B}}
\newcommand{\cC}{\mathcal{C}}
\newcommand{\dD}{\mathcal{D}}
\newcommand{\eE}{\mathcal{E}}
\newcommand{\fF}{\mathcal{F}}

\newcommand{\hH}{\mathcal{H}}

\newcommand{\lL}{\mathcal{L}}
\newcommand{\mM}{\mathcal{M}}
\newcommand{\nN}{\mathcal{N}}
\newcommand{\oO}{\mathcal{O}}
\newcommand{\pP}{\mathcal{P}}

\newcommand{\rR}{\mathcal{R}}
\newcommand{\sS}{\mathcal{S}}
\newcommand{\tT}{\mathcal{T}}

\newcommand{\wW}{\mathcal{W}}

\newcommand{\zZ}{\mathcal{Z}}

\newcommand{\bbC}{\mathbb{C}}

\newcommand{\bbH}{\mathbb{H}}

\newcommand{\bbP}{\mathbb{P}}
\newcommand{\bbQ}{\mathbb{Q}}
\newcommand{\bbR}{\mathbb{R}}

\newcommand{\bbZ}{\mathbb{Z}}

\newcommand{\Hom}{\mathop{\rm Hom}\nolimits}
\newcommand{\vin}{\rotatebox{90}{$\in$}}

\newcommand{\rk}{\mathop{\rm rk}\nolimits}

\newcommand{\Ext}{\mathop{\rm Ext}\nolimits}

\newcommand{\Coh}{\mathop{\rm Coh}\nolimits}

\newcommand{\Imm}{\mathop{\rm Im}\nolimits}

\newcommand{\Ker}{\mathop{\rm Ker}\nolimits}
\newcommand{\Ree}{\mathop{\rm Re}\nolimits}
\newcommand{\Stab}{\mathop{\rm Stab}\nolimits}

\newcommand{\GL}{\mathop{\rm GL}\nolimits}

\begin{document}

\title{Stabilty conditions on degenerated elliptic curves}
\author{Tomohiro Karube}
\date{\today}

\address{Graduate School of Mathematical Sciences, The university of Tokyo, Meguro-ku, Tokyo, 153-8914, Japan.}
\email{karube-tomohiro803@g.ecc.u-tokyo.ac.jp}

\begin{abstract}
We study stability conditions on reducible Kodaira curves obtained from degenerations of elliptic curves. 
We describe connected components of the spaces of stability conditions and compute the groups of deck transformations of those connected components.
\end{abstract}

\maketitle

\setcounter{tocdepth}{1}
\tableofcontents
\tableofcontents

\section{Introduction}

The notion of stability conditions on triangulated categories was introduced by Bridgeland in \cite{Bri:07}. They are inspired by the work in string theory \cite{Dou:02}. 
Bridgland proved a strong deformation property \cite{Bri:07} and he also showed that the spaces of stability conditions satisfying certain properties on the triangulated categories are complex manifolds. Those spaces are called \textit{stability manifolds} and have been studied as invariants of the triangulated categories. 

Bridgeland studied stability conditions on the derived category of a K3 surface $X$ in \cite{Bri:08}. He showed that the forgetful map from the connected component $\Stab^{\dagger}(X)$ of the stability manifold to an open set of the numerical Grothendieck group $\nN(X) \otimes \bbC$ is a covering map, and a subgroup of autoequivalences of $X$ is the group of deck transformations.
In this way, the study of the stability manifold $\Stab(X)$ gives a geometric approach to understanding the group $\mathrm{Auteq}(X)$.

This paper aims to give a relationship between the stability manifold $\Stab(C)$ and the group $\mathrm{Auteq}(C)$ in the case when $C$ is the reducible Kodaira curve.
And we will describe one connected component of the space $\Stab(C)$ and compute the group of deck transformations of the forgetful map. 
In the case of smooth elliptic curves or singular irreducible curves of arithmetic genus one, the spaces of stability conditions are known to be isomorphic to the universal covering space $\widetilde{\GL}^{+}(2,\bbR)$ of $\GL^{+}(2,\bbR)$ \cite{Bri:07,BK:06}. 
However, we will show that the stability manifolds of reducible Kodaira curves are not isomorphic to $\widetilde{\GL}^{+}(2,\bbR)$.

\subsection{Result}
Let $C$ be a reducible singular fiber of a relatively minimal elliptic surface $S$ on the following list.
\begin{itemize}
\item[$(I_N)$] a cycle of $N$ projective lines with $N\geq2$.
\item[$(III)$] two projective lines which meet at one point of order 2.
\item[$(IV)$] three projective lines which meet in one point.
\item[$(I^*_N)$] $N+5$ projective lines corresponding to affine Dynkin diagram $\widetilde{D_{N+4}}$.
\item[$(II^*)$] $8$ projective lines corresponding to affine Dynkin diagram $\widetilde{E_{8}}$.
\item[$(III^*)$] $7$ projective lines corresponding to affine Dynkin diagram $\widetilde{E_{7}}$.
\item[$(IV^*)$] $6$ projective lines corresponding to affine Dynkin diagram $\widetilde{E_{6}}$. 
\item[$(_mI_N)$] $C = mC_N$ where $C_N$ is a curve corresponding to $(I_N)$ with $N\geq2$. 
\end{itemize}
The curve $C$ is called the Kodaira curve since Kodaira gave the classification of singular fibers in the case of  elliptically fibered surfaces over a smooth curve \cite{Kod:63}.
Let $C_1,\ldots,C_n$ be irreducible components of $C$ and $\Theta_i$ be the reduced curve of the irreducible curve $C_i$.


Let $\dD^b(C)$ denote the bounded derived category of coherent sheaves on $C$, and $\Stab(C)$ be the space of stability conditions on $\dD^b(C)$. By definition of stability conditions, there is a forgetful map
\[
\pi : \Stab(C) \rightarrow G(C)^{\vee}\otimes \bbC 
\]
sending a stability condition $\sigma = (Z,\pP)$ to the central charge $Z \in G(C)^{\vee}\otimes \bbC $. Here, $G(C)$ is the Grothendieck group of $\dD^b(C)$.  As in the case of a smooth curve, there are stability conditions on $\dD^b(C)$ corresponding to the slope stability. We will restrict our attention to the connected component $\Stab^{\dagger}(C)\subset \Stab(C)$ containing them. 

We introduce a pairing $\langle -,- \rangle$ on $G(C)^{\vee}\otimes \bbC$ defined as 
\[
\langle E,F \rangle = -\chi(i_*E,i_*F),
\]where $i$ is an embedding to a relatively minimal elliptic surface $S$ and $\chi$ is the Eular characteristic of $E,F$ on $S$.
Let 
\[
V_0 =\{v \in G(C)\otimes \bbC  \mid \langle v,v\rangle = 0\}
\] be the linear subspace of $G(C)\otimes \bbC$.
Let
\[
\pP_0(C) = 
\left\{ Z \in G(C)^{\vee}\otimes\bbC \;\middle|\;
\begin{array}{l}
Z(\delta) \not = 0 \mathrm{\ for \ all\ }\delta \text{ with }\langle \delta ,\delta \rangle = -2\\
\Ree (Z\circ s) \text{\ and\ } \Imm(Z\circ s) \text{\ are\ linearly\ independent\ in\ } V_0^{\vee}
\end{array}
\right\}
\]
be an open set in $G(C)^{\vee}\otimes \bbC$. Here, $s$ denotes the inclusion $V_0 \hookrightarrow G(C)\otimes \bbC$. 
There are two connected components of $\pP_0(C)$, and $\pP_0^{+}(C)$ stands for a connected component containing the central charge $Z = -\chi+\sqrt{-1}\rk_1+\cdots + \sqrt{-1}\rk_n$ defined in Section 4.

 Let $\mathrm{Auteq}_0(C)$ be the subgroup of $\mathrm{Auteq}(C)$ consisting of autoequivalences which act trivially on $G(C)$ and $\mathrm{Auteq}_0^*(C) \subset \mathrm{Auteq}_0(C)$ be the subgroup of autoequivalences preserving the connected component $\Stab^{\dagger}(C)$. 
 
Let
 \begin{align*}
 &\mathrm{Aut}_{\mathrm{tri}}(C) =\{f \in \mathrm{Aut}(C)\mid f_* \mathrm{\ is\ trivial\ in\ }G(C)\},\\
 &\mathrm{Pic}^0_{\mathrm{tri}}(C) =\{\lL \in \mathrm{Pic}^0(C)\mid \lL|_{\Theta_i} \cong \oO_{\Theta_i} \mathrm{\ for\ all\ reduced\ curves\ }\Theta_i \}.
 \end{align*}
  Contrary to the case of K3 surfaces, we remark that any autoequivalences in $\mathrm{Pic}^0_{\mathrm{tri}}(C)\rtimes \mathrm{Aut}_{\mathrm{tri}}(C) $  have the trivial action on $\Stab^{\dagger}(C)$. 
	
The following is the main result of this paper.
\begin{thm}[Theorem \ref{thm:main1}]\label{thm:main}
The connected component $\Stab^{\dagger}(C)$ is mapped onto the open subset $\pP_0^{+}(C)\subset G(C)^{\vee}\otimes \bbC$, and the restricted forgetful map
\[
\pi:\Stab^{\dagger}(C) \rightarrow \pP_0^{+}(C)
\]
is a covering map.

Moreover, the deck transforms of $\pi$ is $\mathrm{Auteq}_0^*(C)/(\mathrm{Pic}^0_{\mathrm{tri}}(C)\rtimes \mathrm{Aut}_{\mathrm{tri}}(C))$.
\end{thm}

Recall that it is important to observe the actions of spherical twists, defined in \cite{ST:01}, on stability manifolds of K3 surfaces \cite{Bri:08}. 
To prove Theorem \ref{thm:main}, we need the spherical twist $T_{\oO_{\Theta_i}(k)}$ associated with the sheaf $\oO_{\Theta_i}(k)$ on the irreducible component $C_i$. 
However, torsion-free sheaves supported on one irreducible component are not spherical objects. 
Indeed, if $C$ is a curve of type $(I_n)$, we can compute
\[\dim \Ext^q_{C}(\oO_{C_i}(k),\oO_{C_i}(k)) = 
\begin{cases}
        {0 \ (q <0 \text{ or }q \text{ is odd} )}\\
        {1\ (q=0)}\\
        {2 \ (q \text{ is even})}.
\end{cases}
\] 
Therefore, we define the restricted spherical twists. We fix a relatively minimal elliptic surface $f\colon S \rightarrow B$ over a quasi-projective curve $B$. We assume that the fiber $S_0$ of a fixed point $0 \in B$ is isomorphic to the singular curve $C$.

\begin{thm}[Theorem \ref{thm:twis}]
Let $G$ be a $(-2)$ curve contained in the fiber $S_0$. Then for any integer $k$, there is a sheaf $\pP_B$ on $S\times_BS$ such that the kernel $\pP$ of a spherical twist $T_{\oO_G(k)}$ is canonically isomorphic to the pushforward of $\pP_B$ along the inclusion $S\times_BS \rightarrow S\times S$.
In particular, the spherical twist $T_{\oO_G(k)}$  induces an autoequivalence $\Phi$ of $\dD^b(S_0)$, and the following diagram commutes
\[\begin{tikzcd}
	{\dD^b(S_0)} & {\dD^b(S_0)} \\
	\dD^b(S) & \dD^b(S),
	\arrow["{j_*}", from=1-1, to=2-1]
	\arrow["\Phi"', from=1-1, to=1-2]
	\arrow["{T_{\oO_G(k)}}", from=2-1, to=2-2]
	\arrow["{j_*}"', from=1-2, to=2-2]
\end{tikzcd}\]
where $j$ is a closed immersion $S_0 \hookrightarrow S$.
\end{thm}

\subsection{Relation to existing work}

There are several works relating Bridgeland stability conditions and stability manifolds on algebraic varieties. 
In \cite{Bri:07}, Bridgeland showed that these spaces of elliptic curves are isomorphic to $\widetilde{\GL}^{+}(2,\bbR)$. For the curves of genus greater than one, these spaces are known to be isomorphic to $\widetilde{\GL}^{+}(2,\bbR)$ \cite{Mac:07}.
For the case of the projective line $\bbP^1$, the situation is different since there is a t-structure whose heart is called Beilinson’s \textit{Kronecker heart} $\bB$ in $\dD^b(\bbP^1)$ \cite{Bei:78}.
On the other hand, the stability manifold $\Stab(\bbP^1)$ is known to be isomorphic to $\bbC^2$ in \cite{Oka:04}. In particular, the spaces of stability conditions on non-singular curves are always simply connected.
Moreover, the stability manifolds of the irreducible curves of arithmetic genus one are known to be isomorphic to $\widetilde{\GL}^{+}(2,\bbR)$ as in the case of smooth elliptic curves.

However, stability conditions on reducible singular curves have not been studied enough.
The stability manifolds of these curves turn out to be more complicated than the smooth curves.
It is worth pointing out that there are many similarities between the case of K3 surfaces and Kodaira curves, such as Theorem \ref{thm:main}.
As in the case of K3 surfaces, we expect the following conjecture:
\begin{conj}
The action of $\mathrm{Auteq}(C)$ on $\Stab(C)$ preserves the connected component $\Stab^{\dagger}(C)$. Moreover $\Stab^{\dagger}(C)$ is simply connected. Thus there is an isomorphism
\[
\mathrm{Auteq}^0(C)/(\mathrm{Pic}^0_{\mathrm{tri}}(C) \rtimes \mathrm{Aut}_{\mathrm{tri}}(C))\cong \pi_1 \pP_0^+(C).
\]
\end{conj}

\subsection{Future directions}
We expect that the results of this paper provide clues to investigating the families of stability conditions defined in \cite{BL:21}.
Let $X \rightarrow S$ be a flat family of projective varieties over some base scheme $S$.
A stability condition on $\dD^b(X)$ over $S$ is a collection of stability conditions on each fiber $X_s$ with certain conditions.
As in the case of the usual stability conditions, there is a natural topology on the set of the families of stability conditions $\Stab(X/S)$ and its space is a complex manifold \cite[Theorem 1.2.]{BL:21}.

In the case when $X \rightarrow S$ is a degeneration of elliptic curves, we will study topological properties of the natural map $\Stab(X/S)\rightarrow \Stab(X_0)$ and relations among the stability conditions of the singular fiber and smooth fibers.

\subsection*{Acknowledgement}
I am grateful to my supervisor Yukinobu Toda for introducing this problem to me, and for helpful discussions. I also thank Tasuki Kinjo for pointing out various errors in the manuscript of this paper.
I was supported by WINGS-FMSP
program at the Graduate School of Mathematical Science, the University of Tokyo.
The content of Section 3 of this paper is partly the same as that of Hayato Arai's later paper\cite{Ari:23}, but it is an independent work.

\section{Preliminaries}

\subsection{Stability conditions}
We review the notion of Bridgeland stability conditions \cite{Bri:07,KS:08} and the deformation property. First, we define stability functions on abelian categories.

Let $\aA$ be an abelian category and $K(\aA)$ be its Grothendieck group.

\begin{defi}
\begin{enumerate}
\item 
A stability function on $\aA$ is a group homomorphism $Z : K(\aA) \rightarrow \bbC$ such that for all $0\not=E \in \aA$
\[
Z(E) \in \bbH \cup \bbR_{<0}.
\]
Here, $\bbH$ is the upper half plane.

\item
Let $Z$ be a stability function on $\aA$. The phase of an object $0 \not = E \in \aA$ is defined to be \[\phi(E) = (1/\pi) \arg Z(E) \in \left( 0,1\right].\]

\item
An object $0 \not = E \in \aA$ is said to be semistable (resp. stable) if every subobject $0 \not = A \subset E$ satisfies $\phi(A) \leq(resp. <) \phi(E)$.
\item $Z$ satisfies the Harder-Narashimhan property (HN property) if for every object $E\in\aA$ there is a filtration
\[
0 = E_0 \subset E_1 \subset \cdots \subset E_{n-1} \subset E_n
\]
whose factors $F_j = E_j / E_{j-1}$ are semistable objects of $\aA$ with
\[
\phi(F_1) >\phi(F_2) > \cdots> \phi(F_n).
\]
\end{enumerate}
\end{defi}
Let $\dD$ be a triangulated category. We define a pre-stability condition and the support property. We fix a finite free abelian group $\Lambda$ with a group homomorphism $v:K(D) \rightarrow \Lambda$. We note that if $\aA$ is the heart of a bounded t-structure on $\dD$ then $K(\aA)$ is isomorphic to $K(D)$.

\begin{defi}
\begin{enumerate}
\item 
A pre-stability condition $\sigma = (Z, \aA)$ on $\dD$ with respect to $\Lambda$ consists of the heart of a bounded t-structure $\aA$ and a stability function $Z$ on $\aA$ which factors through $v:K(\aA) \cong K(\dD) \rightarrow \Lambda$ and satisfies the HN-property. 
\item
Let $\lVert-\rVert$ be a norm on $\Lambda_{\bbR} = \Lambda\otimes \bbR$. A pre-stability condition $\sigma = (Z, \aA)$ on $\dD$ with respect to $\Lambda$ satisfies the support property if the following holds:
\[
\sup \left\{  \frac{\lVert v(E)\rVert}{\lvert Z(E) \rvert}\mid E \rm{\ is\ semistable\ in\ }\aA \right\} <\infty.
\]
\item 
A stability condition is a pre-stability condition satisfying the support property.

\end{enumerate}
\end{defi}

\begin{lem}[{\cite[Lemma A.4.]{BMS:16}}]
A pre-stability condition $\sigma = (Z,\aA)$ with respect to $\Lambda$ satisfies the support condition if and only if there exists a quadratic form $Q$ on the vector space $\Lambda_{\bbR}$ such that
\begin{itemize}
\item for any semistable objects $E \in \aA$
\[
Q(v(E)) \geq 0.
\]
\item $Q$ is negative definite on $\Ker Z \subset \Lambda_{\bbR}$.
\end{itemize}
\end{lem}

\begin{defi}
A slicing $\pP =\{\pP(\phi)\}_{\phi \in \bbR}$ of $\dD$ consists of full addictive subcategories $\pP(\phi) \subset \dD$ for each $\phi$, satisfying:
\begin{enumerate}
\item for all $\phi \in \bbR$, $\pP(\phi+1) = \pP(\phi)[1]$.
\item if $\phi_1 > \phi_2$ and $E_j \in \pP(\phi_j)$, then $\Hom(E_1,E_2) = 0$.
\item for every $E \in \dD$ there exists a finite sequence 
\[
0 = E_0 \rightarrow E_1 \rightarrow \cdots E_m =E
\]
such that the cone $\mathrm{Cone}(E_i \rightarrow E_{i+1})$ lies in $\pP(\phi_i)$ with $\phi_0 >\phi_1 >\cdots > \phi_{m-1}$.
 
\end{enumerate}

\end{defi}
Let $\sigma =(Z,\aA)$ be a stability condition on $\dD$. For each real number $\phi \in (0,1]$, full addictive category $\pP(\phi)$ is defined as follows:
\[
\pP(\phi) :=\{E \in \aA \mid E \text{ is semistable of the phase } \phi  \}.
\]

For each real number $\phi \in \bbR$,
\[
\pP(\phi) = \pP (\phi-k ) [k]
\]where $k = \lfloor \phi \rfloor$. 
 It follows from the support property that each category $\pP(\phi)$ is of finite length, so each semistable object has a Jordan-Holder filtration.
For a non-zero object $E$, we write $\phi_{\sigma}^{+} (E) = \phi_1$ and $\phi_{\sigma} ^{-}(E) = \phi_m$. The mass of $E$ is defined by $m(E) = \sum |Z(A_i)|$.
Then, the set $\Stab_{\Lambda}(\dD)$ of stability conditions with respect to $\Lambda$ on $\dD$ has a natural topology induced by the following metric
\[
d(\sigma_1,\sigma_2) = \sup_{0\not = E \in \dD} \left\{ \lvert\phi_{\sigma_1}^{-}(E) - \phi_{\sigma_2}^{-}(E)\rvert,\lvert\phi_{\sigma_1}^{+}(E) - \phi_{\sigma_2}^{+}(E)\rvert,\lvert\log \frac{m_{\sigma_1}(E)}{m_{\sigma_2}(E)}\rvert \right\}.
\]
This topology makes the forgetful map continuous
\[
\begin{array}{ccc}
\pi:\Stab(\dD)&\rightarrow& \Hom(\Lambda,\bbC)\\
\vin &&\vin\\
(Z,\aA)&\mapsto&Z.
\end{array}
\]
It follows from this result called the deformation property of stability conditions that the forgetful map $\pi$ is a local homeomorphism.

\begin{thm}[{\cite[Lemma A.5.]{BMS:16}},{\cite[Theorem 1.2.]{Bay:19}}]\label{thm:deformation}
Let $Q$ be a quadratic form on $\Lambda \otimes\bbR$ and assume that the stability condition $\sigma = (Z ,\pP)$ satisfies the support property with respect to $Q$. 
Consider the open set of $\Hom(\Lambda,\bbC)$ consisting of central charges whose kernel is negative-definite with respect to $Q$, and let $U$ be the connected component containing $Z$.
Then:
\begin{enumerate}
\item There is an open neighborhood $\sigma \in U_{\sigma} \subset \Stab_{\Lambda}(\dD)$ such that the restriction $\zZ : U_{\sigma} \rightarrow U$ is a covering map.
\item All stability conditions in $U_{\sigma}$ satisfy the support property with respect to $Q$.
\end{enumerate}
\end{thm}

\section{Spherical twists along line bundles on $(-2)$ curves}

First, we recall the notion of spherical objects and spherical twists.

\subsection{Spherical twists}

Let $X$ be an $n$-dimensional smooth quasi-projective variety.

\begin{defi}
An object $\eE \in \dD ^b(X)$ with compact support is called spherical if 
\begin{enumerate}
\item $\eE \otimes \omega_X \cong \eE$.
\item $\bbR \Hom(\eE,\eE) \cong H^*(S^{n},\bbC)$.
\end{enumerate}
where $S^{n} $ is the real sphere of dimension $n = \dim(X)$.

\end{defi}

Let $\eE$ be a spherical object and consider the mapping cone
\[
\pP _{\eE} = \mathrm{Cone} \left(\pi_1^*\eE^{\vee}\otimes \pi_2^* \eE \rightarrow \oO_{\Delta} \right).
\]
Here, $\oO_{\Delta}$ is the structure sheaf of the diagonal $\Delta \subset X\times X$ and $\pi_i$ is the $i$-th projection $\pi_i \colon  X\times X \rightarrow X$. The homomorphism $ \pi_1^*\eE^{\vee}\otimes \pi_2^* \eE \rightarrow \oO_{\Delta}$ is an evaluation map given as the composition of the restriction
\begin{align}
\pi_1^*\eE^{\vee}\otimes \pi_2^* \eE \longrightarrow \iota_* \iota^* (\pi_1^*\eE^{\vee}\otimes \pi_2^* \eE) = \iota_*(\eE^{\vee}\otimes \eE), \label{def:unit}
\end{align}
 where $\iota \colon X \rightarrow \Delta \subset X\times X$, and the trace map $\iota _* \mathrm{tr}$
\begin{align}
\iota _* \rm{tr}\colon \iota_*(\eE^{\vee}\otimes \eE) \longrightarrow \iota_*(\oO_X) = \oO_{\Delta}.\label{def:trace}
\end{align}

\begin{prop-defi}[\cite{ST:01}]
Let $\eE$ be a spherical object in $\dD^b(X)$. The following functor 
\[
T_{\eE}\coloneqq \Phi_{\pP_{\eE}} \colon \dD^b(X) \rightarrow \dD^b(X)
\]
is an autoequivalence, called the spherical twist.
\end{prop-defi}

\subsection{Spherical twists along line bundles on $(-2)$ curves}

Let $S \rightarrow B$ be a relatively minimal smooth elliptic surface such that $B$ is a quasi-projective curve. We consider a $(-2)$ curve $G$ contained in a fiber $S_0$. 
We will claim that the spherical twist along $\oO_G(k)$ induces an autoequivalence of the derived category of each fiber. To show this statement, we use the following.

\begin{thm}[{cf.\cite[Proposition 2.15.]{HLS:09},\cite[Lemma 2.5.]{Ueh:15}}]
Let $\pi\colon X \rightarrow C$ and $\pi^{\prime}\colon Y \rightarrow C$ be flat projective morphisms between smooth varieties, and take a point $c\in C$. Suppose $\pP$ is an object in $\dD^b(X\times_C Y)$. Consider two functors $\Phi = \Phi_{\pP}\colon \dD^b(X) \rightarrow \dD^b(Y)$ and $\Psi = \Phi_{\pP|_{X_c\times Y_c}} \colon \dD^b(X_c) \rightarrow \dD^b(Y_c)$. Then we have:
\begin{itemize}
\item These functors satisfy ${i_{Y}}_*\circ \Phi\cong \Psi \circ {i_X}_*$. Here $i_{X}$ and $i_Y$ denote the inclusions from the fibers $X_c,Y_c$.
\item If $\Phi$ is an equivalence, then $\Psi$ is also an equivalence.
\end{itemize}
\end{thm}

The $i$-th projection $A_1 \times A_2  \rightarrow A_i$ will be denoted by ${\pi_i} _{A_1 A_2}$.
We show that the kernel $\pP_{\oO_G(k)}$ of the spherical twist $T_{\oO_G(k)}$ can be obtained by the pushforward of an object in $\dD^b(S\times_B S)$.
By the Grothendieck-Verdier duality, we have an isomorphism
\begin{align}\label{lem:dual}
(i_*\oO_G(k))^{\vee} \cong i_*\oO_G(-k-2)[-1]
\end{align}where $i$ is the inclusion $G\hookrightarrow S$.

\begin{lem}
Let $j_{GG\rightarrow SS}\colon G\times G \rightarrow S\times S$ be a closed immersion.
Then,
\[
{\pi_1}_{S S}^* \oO_G(k) \otimes {\pi_2}_{S S}^*( \oO_G(k))^{\vee} \cong {j_{GG\rightarrow SS}}_* ({\pi_1}_{GG}^*\oO_G(k) \otimes {\pi_2}_{GG}^*\oO_G(-k-2))[-1].
\]
\end{lem}
\begin{proof}
Consider the following diagram:
\[\begin{tikzcd}
	&& {G\times G} \\
	G & {G \times S} && G \\
	S & {S\times S} & S
	\arrow["{j_{GS\rightarrow SS}}", hook, from=2-2, to=3-2]
	\arrow["{{\pi_1}_{GS}}"', from=2-2, to=2-1]
	\arrow["{{\pi_1}_{SS}}", from=3-2, to=3-1]
	\arrow["i",hook, from=2-1, to=3-1]
	\arrow["{{\pi_2}_{SS}}", from=3-2, to=3-3]
	\arrow["{j_{GG\rightarrow GS}}", hook, from=1-3, to=2-2]
	\arrow["{{\pi_2}_{GG}}", from=1-3, to=2-4]
	\arrow["i",hook, from=2-4, to=3-3]
\end{tikzcd}\]
Then,
\begin{align*}
{\pi_1}_{SS}^* \oO_G(k) \otimes {{\pi_2}_{SS}}^*( \oO_G(k))^{\vee} &  \cong {j_{GS \rightarrow SS}}_* ({\pi_1}_{GS}^* \oO_G(k)) \otimes {\pi_2}_{SS}^* ( \oO_G(k))^{\vee}\\
&\cong {j_{GS \rightarrow SS}}_* ({\pi_1}_{GS}^* \oO_G(k) \otimes {\pi_2}_{GS}^* ( \oO_G(k))^{\vee}) \\
& \cong {j_{GS \rightarrow SS}}_* ({\pi_1}_{GS}^* \oO_G(k) \otimes {j_{GG \rightarrow GS}}_* {\pi_2}_{GG}^* ( \oO_G(-k-2)))[-1] \\
&\cong {j_{GG\rightarrow SS}}_* ({\pi_1}_{GG}^*\oO_G(k) \otimes {\pi_2}_{GG}^*\oO_G(-k-2))[-1].
\end{align*}
Here, the first isomorphism follows from the flat base change, the second isomorphism follows from the projection formula, the third isomorphism follows from Isomorphism (\ref{lem:dual}) and the last isomorphism follows from the projection formula.
\end{proof}

Next, we construct a map
\begin{align}\label{map:1}
{i_{GG}}_*({\pi_1}_{GG}^*\oO_G(k) \otimes {\pi_2}_{GG}^*\oO_G(-k-2))[-1] \rightarrow \Delta_{S *} \oO_S.
\end{align}
where each map is defined in the following diagram:
\[\begin{tikzcd}
	G & S \\
	{G \times G} & {S \times_B S} \\
	&& {S\times S}
	\arrow["i", from=1-1, to=1-2]
	\arrow["{\Delta_G}"', from=1-1, to=2-1]
	\arrow["{\Delta_S}"{description}, from=1-2, to=2-2]
	\arrow["{i_{GG}}"', from=2-1, to=2-2]
	\arrow["j"', from=2-2, to=3-3]
	\arrow["{\iota_S}"{description}, from=1-2, to=3-3]
\end{tikzcd}\]

Map (\ref{map:1}) is given as the composition of the pushforward of the restriction
\begin{align*}
{i_{GG}} _* \eta  \colon {i_{GG}}_*({\pi_1}_{GG}^*\oO_G(k) \otimes {\pi_2}_{GG}^*\oO_G(-k-2)) \rightarrow &{i_{GG}}_* {\Delta_G}_*(\oO_G(-2))\\
&\cong{i_{GG}}_* {\Delta_G}_*{\Delta_G}^*({\pi_1}_{GG}^*\oO_G(k) \otimes {\pi_2}_{GG}^*\oO_G(-k-2))\\
\end{align*}
and the following morphism
\[
{\Delta_S}_*f\colon {i_{GG}}_* {\Delta_G}_*(\oO_G(-2))[-1] \cong {\Delta_S}_* i_*\oO_G(-2)[-1]\rightarrow {\Delta_S}_*\oO_S
\]
where $f$ is the non-zero morphism defined by the non-zero unique vector in $\Hom(i_*\oO_G(-2)[-1],\oO_S) \cong \bbC$, and it is determined up to scalar.

\begin{lem}
There is a following commutative diagram:
\[\begin{tikzcd}
	{{{\pi_1}_{SS}}^*\oO_G(k)\otimes {{\pi_2}_{SS}}^*\oO_G(k)^{\vee}} & {{\iota_S}_*{\iota_S}^*({{\pi_1}_{SS}}^*\oO_G(k)\otimes {{\pi_2}_{SS}}^*\oO_G(k)^{\vee})} & {\oO_{\Delta_S}} \\
	{{j_{GG \rightarrow SS}}_*({{\pi_1}_{GG}}^*\oO_G(k) \otimes {{\pi_2}_{GG}}^*\oO_G(-k-2))[-1] } & {j_*{\Delta_S}_* i_*\oO_G(-2)[-1]}& {\oO_{\Delta_S}}
	\arrow[from=1-1, to=2-1]
	\arrow["{\eta_S}",from=1-1, to=1-2]
	\arrow["{i_{GG}} _* \eta",from=2-1, to=2-2]
	\arrow[from=1-2, to=2-2]
	\arrow["j_*{\Delta_S}_*f",from=2-2, to=2-3]
	\arrow["{\iota_S \rm{tr}}", from=1-2, to=1-3]
	\arrow["id",from=1-3, to=2-3]
\end{tikzcd}\]
\end{lem}
\begin{proof}
First, we define a map 
\[{{\iota_S}_*{\iota_S}^*({\pi_1}_{SS}^*\oO_G(k)\otimes {\pi_2}_{SS}^*\oO_G(k)^{\vee})} \rightarrow {j_*{\Delta_S}_* i_*\oO_G(-2)[-1]}\cong {\iota_S}_* \oO_G(-2)[-1]. \]
Since we have an isomorphism in $\dD^b(G)$
\[
i_*i^*{{\iota_S}^*({\pi_1}_{SS}^*\oO_G(k)\otimes {\pi_2}_{SS}^*\oO_G(k)^{\vee})}\cong \oO_G(-2)[-1], 
\]
there is a natural map $g\colon {{\iota_S}_*{\iota_S}^*({\pi_1}_{SS}^*\oO_G(k)\otimes {\pi_2}_{SS}^*\oO_G(k)^{\vee})} \rightarrow {\iota_S}_*\oO_G(-2)[-1]$.

Since $g \circ \eta_S$ is non-zero and $\Hom({{\pi_1}_{SS}}^*\oO_G(k)\otimes {{\pi_2}_{SS}}^*\oO_G(k)^{\vee},{\iota_S}_*\oO_G(-2)[-1]) \cong \bbC$, the left square of the diagram commutes.
The right square of the diagram can be obtained from the pushforward of the following diagram in $\dD^b(S)$:
\[\begin{tikzcd}
	 {{\iota_S}^*({{\pi_1}_{SS}}^*\oO_G(k)\otimes {{\pi_2}_{SS}}^*\oO_G(k)^{\vee})} & {\oO_{S}} \\
 \oO_G(-2)[-1] & {\oO_{S}}
	\arrow["g",from=1-1, to=2-1]
	\arrow["f",from=2-1, to=2-2]
	\arrow["{ \rm{tr}}", from=1-1, to=1-2]
	\arrow["id",from=1-2, to=2-2]
\end{tikzcd}\]
Since $f\circ g$ is non-zero and $\Hom({\iota_S}^*({{\pi_1}_{SS}}^*\oO_G(k)\otimes {{\pi_2}_{SS}}^*\oO_G(k)^{\vee}),{\oO_{S}}) \cong \bbC$, the diagram commutes.
\end{proof}

The result of Theorem \ref{thm:twis} follows from the above lemma.
\begin{thm}\label{thm:twis}
Let $j$ denote the closed immersion $S\times_B S\hookrightarrow S\times S$. There is an object $\pP_{G,k}$ in $\dD^b(S\times_B S)$ such that the kernel $\pP_{\oO_G(k)}$ of a spherical twist $T_{\oO_G(k)}$ is isomorphic to $j_*\pP_{G,k}$.
In particular, the spherical twist $T_{\oO_G(k)}$  induces an autoequivalence $\Phi\colon \dD^b(S_0) \to \dD^b(S_0)$, and the following diagram commutes:
\[\begin{tikzcd}
	{\dD^b(S_0)} & {\dD^b(S_0)} \\
	\dD^b(S) & \dD^b(S)
	\arrow["{{j_0}_*}", from=1-1, to=2-1]
	\arrow["\Phi"', from=1-1, to=1-2]
	\arrow["{T_{\oO_G(k)}}", from=2-1, to=2-2]
	\arrow["{{j_0}_*}"', from=1-2, to=2-2]
\end{tikzcd}\]
where $j_0$ is a closed immersion $S_0 \hookrightarrow S$.

\end{thm}

We call these autoequivalences of $S_0$ the \textit{restricted spherical twists}. 
 For simplicity of notation, we use the same latter $T_{\oO_G(k)}$ for the autoequivalence obtained from the restriction of the spherical twist $T_{\oO_G(k)}$.

 \begin{rmk}
Let $f \colon S \rightarrow S^{\prime}$ be a birational map contracting the $(-2)$ curve $G$.
By the same argument, we conclude the kernel $\pP_{\oO_G(k)}$ is isomorphic to
the pushforward of a torsion-free sheaf $\pP \in \Coh(S\times_{S^{\prime}}S)$ of rank one.
Indeed, the sheaf $\pP$ can be obtained from an exact sequence in $\Coh(S\times_{S^{\prime}}S)$
\[
	0 \rightarrow \oO_{\Delta_S} \rightarrow \pP \rightarrow {i_{S^{{\prime}}}}_*({{\pi_1}_{GG}}^*\oO_G(k) \otimes {{\pi_2}_{GG}}^*\oO_G(-k-2)) \rightarrow 0
\] where $i_{S^{\prime}}$ is a closed immersion $G \times G \hookrightarrow S \times_{S^{\prime}}S$.

 \end{rmk}

\section{Stability conditions on $C$}
In this section, 
we construct a stability condition corresponding to the "slope stability" on a non-irreducible curve which is obtained as a singular fiber of an elliptic surface.
Let $S\rightarrow B$ be a relatively minimal elliptic surface with a quasi-projective curve $B$ and $C$ be a singular fiber of $S$. 
We consider the singular curve $C$ in the following list as reducible \textit{Kodaira curves}.
\begin{itemize}
\item[$(I_N)$] a cycle of $N$ projective lines with $N\geq2$.
\item[$(III)$] two projective lines which meet at one point of order 2.
\item[$(IV)$] three projective lines which meet in one point.
\item[$(I^*_N)$] $N+5$ projective lines corresponding to affine Dynkin diagram $\widetilde{D_{N+4}}$.
\item[$(II^*)$] $8$ projective lines corresponding to affine Dynkin diagram $\widetilde{E_{8}}$.
\item[$(III^*)$] $7$ projective lines corresponding to affine Dynkin diagram $\widetilde{E_{7}}$.
\item[$(IV^*)$] $6$ projective lines corresponding to affine Dynkin diagram $\widetilde{E_{6}}$. 
\item[$(_mI_N)$] $C = mC_N$ where $C_N$ is a curve corresponding to $(I_N)$ with $N\geq2$. 
\end{itemize}

The above list does not contain all possible singular fibers of elliptic surfaces, and it is not our purpose to study the following irreducible curves:
\begin{itemize}
\item[$(I_0) $] a smooth elliptic curve.
\item [$(I_1) $] a rational curve with one node.
\item [$(II) $] a rational curve with one cusp.
\item[$(_mI_0) $] $C = mD$ where $D$ is a smooth elliptic curve.
\item[$(_mI_1) $] $C = mD$ where $D$ is a rational curve with one node.
\end{itemize}

Let $C_1,..,C_n$ denote irreducible components of $C$. We will denote by $\Theta_i$ the reduced curve of $C_i$.
Let $E$ be a coherent sheaf on $C$ and denote $E_{C_i} = E|_{C_i}/(\rm{torsion})$. 
For any pure dimension one sheaf $E$, there are two exact sequences
\begin{align}
&0 \rightarrow E^{C_i} \rightarrow E \rightarrow E_{C_i} \rightarrow 0,\\
&0 \rightarrow E \rightarrow E_{C_1}\oplus \cdots \oplus E_{C_n} \rightarrow T \rightarrow 0,
\end{align}
where $E^{C_i}$ is supported on $C_1\cup \cdots \cup C_{i-1}\cup C_{i+1}\cup \cdots \cup C_{n}$, and $T$ is a torsion sheaf whose support is contained in the set of intersections of $C_1,\ldots,C_n$.
From now on, $\Theta^{\mathrm{sing}}$ denotes the set of intersections.

\begin{prop}[{\cite[Proposition 4.3.]{MP:17}}]
Let $C$ be a connected and projective curve and $G(C)$ be the Grothendieck group of the curve $C$.
If every irreducible component $C_i$ of $C$ is isomorphic to $m_i\bbP^1$ for some $m_i$, 
then the group $G(C)$ is free group of  rank $n+1$, that is $G(C) \cong \bbZ^{n+1}$.
Moreover, the group $G(C)$ is generated by 
\[
[\oO_{\Theta_1}],\ldots,[\oO_{\Theta_n}],[\oO_x].
\]
\end{prop}
Let $\rk_i(E)$ denote the component of $[E]$ corresponding to $[\oO_{\Theta_i}]$. 
By above proposition, we have the isomorphism defined by
\[
\begin{array}{ccc}
v\colon G(C)&\longrightarrow &\bbZ^{n+1} \\
\vin &&\vin\\
 E&\longmapsto& ( \chi(E),\rk_1 E,\ldots,\rk_n E).
\end{array}
\]
Here, $\chi(E)$ is the Euler characteristic of $E$.
Since $G(C)$ is of finite rank, we fix the above isomorphism $v$, and we consider stability conditions with respect to $(v,\bbZ^{n+1})$. To simplify notation, we write $[E]$ instead of the class $v(E)$ for each object $E \in \dD^b(C)$

\subsection{Constructing stability conditions corresponding to the slope stability}

First, we take $z_1,\ldots,z_n \in \bbH $ and consider a group homomorphism $Z : G(C) \rightarrow \bbC$ defined by
\[
Z(E) = -\chi(E) + z_1\rk_1E + \cdots+ z_n\rk_nE.
\]
It is clear that $Z$ defines a stability function on $\Coh(C)$. We will show that $(Z ,\Coh(C))$ is a stability condition.

\begin{prop}
For the above stability function $Z$, $(Z,\Coh(C))$ satisfies HN property.
\end{prop}
\begin{proof}
Applying the criterion of \cite[Theorem 2.4.]{Bri:07}, it is enough to show the following two conditions.
\begin{enumerate}
\item There are no infinite sequences in $\Coh(C)$
\[
\cdots \subset E_{j+1}\subset E_{j} \subset \cdots \subset E_1
\]
with $\phi(E_{j+1}) > \phi (E_j) $ for all $j$.

\item  There are no infinite sequences in $\Coh(C)$
\[
E_1 \twoheadrightarrow E_2 \twoheadrightarrow \cdots E_j \twoheadrightarrow E_{j+1}
\]
with $\phi(E_j) > \phi(E_{j+1})$ for all $j$.
\end{enumerate} 
Since $\Coh(C)$ is Noetherian, Condition (2) is satisfied.

Next,we show Condition (1).
Suppose we have a sequence in $\Coh(C)$
\[
\cdots \subset E_{j+1} \subset \cdots \subset E_1
\]
with $\phi(E_{j+1}) > \phi (E_j) $ for all $j$.
Since $\rk_i(E_j)$ is non-negative, we can assume $\rk_i E_j$ is constant. Since $\phi(E_{j+1}) > \phi(E_{j})$, we have $\chi(E_{j+1})  >\chi(E_{j})$.
However, we have a following exact sequence
\[
0 \rightarrow E_{j+1} \rightarrow E_{j} \rightarrow T \rightarrow 0
\]
where $T$ is a torsion sheaf. So, $\chi(T) >0$. But, this contradicts $\chi(E_{j+1})  >\chi(E_{j})$.
\end{proof}

We will show the following proposition in Proposition \ref{section61}.
\begin{prop}\label{prop:supp porp}
Let $Z\colon G(C) \to \bbC$ be a group morphism defined by
\[
Z(E) = -\chi(E) + z_1\rk_1E + \cdots+ z_n\rk_nE.
\]
Then, a pre-stability condition $(Z,\Coh(C))$ satisfies the support property.
\end{prop}

Let $V(C)$ denote the subset 
\[
\{(Z,\Coh(C))\mid Z = -\chi(E) + z_1\rk_1E + \cdots+ z_n\rk_nE ,z_1,\ldots,z_n \in \bbH \} \subset \Stab(C).
\]
Since $V(C)$ is isomorphic to $\bbH^n$, $V(C)$ is connected.
We write $\Stab^{\dagger}(C)$ for a connected component containing the set $V(C)$.
This proposition characterizes all the stability conditions constructed above.

\begin{prop}\label{prop:V(C)}
Let $\sigma=(Z, \pP) \in Stab(C)$ be a stability condition.
Then $\sigma$ lies in $V(C)$ if and only if all skyscraper sheaves $\oO_x$ are stable of phase one and $Z(\oO_x) =-1$.
\end{prop}
\begin{proof}
First, we show that all stable objects in $\pP((0,1])$ are coherent sheaves.
Suppose that $E$ is a stable object of phase $\phi$ for some $0<\phi$.
Since all skyscraper sheaves $\oO_x$ are stable of phase one, we have $\Hom(E,\oO_x[i]) =0 $ for $i\leq -1$. 
Write $m := \max \{i\mid \hH^i(E) \not =0\}$. Here, $\hH^i(E)$ denotes the $i$-th cohomology sheaf of  $E$.
There is an object $E^{<m}$ and a distinguished triangle
\[
E^{<m}\rightarrow E \rightarrow \hH^m(E)[-m] \rightarrow E^{<m}[1].
\]
Since $\Hom(E^{<m}[1],\oO_x[-m]) =0$, and if $\Hom(\hH^m(E),\oO_x) \not=0$, then  $\Hom(E,\oO_x[-m]) \not=0$. Thus, the cohomology sheaf $\hH^i(E)$ vanishes for $i \geq 1$.

Next, let $E$ be a stable object of phase $\phi$ with $\phi<1$. 
Take $\oO_D$ be a structure sheaf of a divisor $D$ corresponding to a point $x\in C\backslash \Theta^{\mathrm{sing}}$.
Since the structure sheaf $\oO_D$ of $x$ can be obtained from extensions of the skyscraper sheaf $\oO_x$, 
$\oO_D$ lies in $\pP((0,1])$, and $\sigma$-semistable of phase one.
Thus, we have $\Hom(\oO_D,E[i])$ for $i\leq0$. 
Since the local ring $\oO_D$ is a perfect object for each $x$ $\not \in  \Theta^{\mathrm{sing}}$.
Serre duality gives
\[
\Hom(E,\oO_D[j]) = \Hom(\oO_D,E[1-j]) = 0 
\]for $j\geq 1$.
Applying a similar argument as above, we conclude that the cohomology sheaf $\hH^i(E)$ is supported on $\Theta^{\mathrm{sing}}$ for $i\leq -1$.
 Assume $\hH^i(E)$ is non-zero for some $i\leq-1$. Define $m := \min \{i\mid \hH^i(E) \not =0\}\leq-1$. 
 Since $\hH^m(E)$ is a torsion sheaf, there is a non-zero morphism $\oO_x[-m] \rightarrow E$ for some $x\in \Theta^{\mathrm{sing}}$. This contradicts the fact that $E$ is stable. Therefore, we have $\hH^i(E)=0$ for $i\leq-1$. 

If $E$ is stable of phase one and $E$ is not a skyscraper sheaf, there cannot exist any non-zero morphisms $E \rightarrow \oO_x$ or $\oO_x \rightarrow E$. It follows from the same argument that the cohomology sheaf $\hH^i(E)$ vanishes for $i\leq-1$.

Therefore, we have $\pP((0,1]) \subset \Coh(C)$. Since both $\pP((0,1])$ and $\Coh(C)$ are the hearts of bounded t-structures, we have $\pP((0,1]) = \Coh(C)$. Consider sheaves $\oO_{\Theta_1}(-1),\ldots,\oO_{\Theta_n}(-1)$, the central charge $Z$ is defined by 
\[
Z(E) = -\chi(E) + z_1\rk_1(E) + \cdots + z_n\rk_n(E),
\]
where $z_i = Z(\oO_{\Theta_i}(-1)) \in \bbH$. Thus, $(Z,\pP)$ lies in $V(C)$.

\end{proof}

\begin{defi}
Let $U(C) \subset \Stab^{\dagger}(C)$ be the open subset consisting of stability conditions $\sigma = (Z,\pP) \in \Stab^{\dagger}(C)$ such that all skyscraper sheaves $\oO_x$ are stable of the same phase in $\sigma$.
\end{defi}

By the previous proposition and the definition, we have $U(C) = V(C)\cdot \widetilde{\GL}^{+}(2,\bbR)$.

\subsection{Constructing stability conditions in the boundary of $U(C)$}

We give an example of a construction of a stability condition in the boundary of $U(C)$. 
Consider the following central charge 
\begin{align}
Z_{r_1}(E) = -\chi(E) + r_1\rk_1(E)+z_2\rk_2(E)+\dots + z_n \rk_n(E) \label{def:CC}
\end{align}
where $r_1 \in \bbR , z_2,...,z_n \in \bbH$.
We will show that the corresponding t-structure is obtained from $\Coh(C)$ by tilting with respect to a torsion pair.

\begin{defi}
 A torsion pair in an abelian category $\aA$ is a pair of full subcategories $(\tT , \fF)$ of $\aA$ which satisfies the following conditions.
 \begin{itemize}
 \item $\Hom(T,F) =0$ for $T \in \tT,F \in \fF$.
 \item Every object $E \in \aA$ fits into a short sequence
 \[
 0 \rightarrow T \rightarrow E \rightarrow F \rightarrow 0
 \]
 for $T \in \tT ,F \in \fF$.
 \end{itemize}
\end{defi} 

Given a triangulated category $\dD$ and a heart $\aA$, for a torsion pair $(\tT,\fF)$ on $\aA$, there is a new heart of a bounded t-structure defined by
\[
\aA ^{\#} = \{E \in \dD \mid \hH^i(E) = 0 \ \mathrm{for} \  i \not \in \{-1,0\},\hH^{-1}(E) \in \fF,\hH^0(E) \in \tT .\}.
\]
For $\sS \subset \dD(X)$, we denote the extension closure of $S$ by $\langle \sS \rangle_{ex}$.

\begin{lem}
The pair of subcategories
\begin{align*}
&\fF_k = \langle\oO_{\Theta_1}(l)\mid l\leq k \rangle_{ex}\\
&\tT_k = {}^{\perp}\fF_k
\end{align*}
defines a torsion pair in $\Coh(C)$ for any $k \in \bbZ$.
\end{lem}
\begin{proof}
Let $E$ be a coherent sheaf on $C$. We show there are objects $T\in\tT_k$ and $F \in \fF_k$ such that satisfies a short exact sequence
\[
0 \rightarrow T \rightarrow E \rightarrow F \rightarrow 0.
\]
Since every sheaf in $\fF_k$ is torsion-free, all torsion sheaves lie in $\tT_k$. Hence, we can assume $E$ is torsion-free.
There is a short exact sequence 
\[
0 \rightarrow E^{C_1} \rightarrow E \rightarrow E_{C_1} \rightarrow 0
\]
where we denote $E_{C_1} = (E|_{C_1}/$torsion$)$ and $E^{C_1} = \Ker(E \rightarrow E_{C_1})$.
Since $E^{C_1}$ is supported on $C_2\cup C_3\cup \dots \cup C_n$, $E^{C_1} \in \tT_k$. Therefore, we can also assume $E=E_{C_1}$. 

There exists a sequence
\[
0=E_N \hookrightarrow \cdots \hookrightarrow E_1 \hookrightarrow E_0 = E
\]whose factors $A_i =E_i/E_{i+1}$ are sheaves on $\Theta_1$.
Since $ E_i/E_{i+1}$ is a direct sum of line bundles on $\Theta_1$, there exist sheaves $A_i^{\leq k}\in \fF_k$ and $A_i^{>k} \in \tT_k$ on $\Theta_1$ such that they satisfy
\[
A_i \cong A_i^{\leq k}\oplus A_i^{>k}.
\]
Consider the surjective $E \cong E_1 \rightarrow A_1^{\leq k} $.
There are exact sequences
\[\begin{tikzcd}
	0 & {E_2} & {E_1} & {A_1} & 0 \\
	0 & {E_2^{\prime}} & {E_1} & {A_1^{\leq k}} & 0
	\arrow[from=1-2, to=1-3]
	\arrow[from=1-3, to=1-4]
	\arrow[from=1-1, to=1-2]
	\arrow[from=2-1, to=2-2]
	\arrow[from=2-2, to=2-3]
	\arrow[from=2-3, to=2-4]
	\arrow[from=2-4, to=2-5]
	\arrow[from=1-2, to=2-2]
	\arrow[from=1-3, to=2-3]
	\arrow[from=1-4, to=1-5]
	\arrow[from=1-4, to=2-4]
\end{tikzcd}\]where $E_2^{\prime}$ is the kernel.

By the snake lemma, $E_2^{\prime}$ is a extension of $E_2$ and $A_1^{>k}$.
Repeating this argument, we have a sequence
\[
0=E^{\prime}_N \hookrightarrow \cdots \hookrightarrow E^{\prime}_1 \hookrightarrow E_0 = E
\]whose factors $E^{\prime}_i/E^{\prime}_{i+1}$ are isomorphic to $A_i^{\leq k}$ on $\Theta_1$.
Here, the subsheaf $E^{\prime}_i$ is obtained from an extension $E_i$ and $A_1^{> k}\ldots A_i^{>k}$.
We have the inclusion $E^{\prime}_N\hookrightarrow E$. Its cokernel is obtained from extensions of $A_1^{\leq k},\ldots,A_N^{\leq k} $, and the lemma follows.

\end{proof}
The proof above give more, namely any torsion-free sheaves in $\tT_k$ which supported on $C_1$ lies in the subcategory $\langle\oO_{\Theta_1}(l) \mid l>k\rangle_{ex}$.
We will denote by $\aA_k$ the heart of the t-structure obtained from tilting with respect to the torsion pair $(\tT_k,\fF_k)$.

\begin{prop}
The group morphism $Z_{r_1} \colon G(C) \rightarrow \bbC$ defines a stability function on $\aA_k$ if $k -1< r_1< k $.
\end{prop}
\begin{proof}
Since $\aA_k = \langle \fF_k[1],\tT_k \rangle_{ex}$, we need to show the following two conditions:
\begin{enumerate}
\item For each object, $T \in \tT_k$, $Z(T) \in \bbH\cup\bbR_{<0}$.
\item  For each object, $F \in \fF_k$, $Z(F) \in\bbR_{>0}$.
\end{enumerate} 
Let $T$ be a sheaf in $\tT_k$.
Assume $\rk_iT = 0$ for $i = 2,3,\ldots ,n$ and $\rk_1(T) >0$.
Since each torsion sheaf $E$ satisfies $Z(E)\in \bbR_{<0}$, we can also assume $T$ is a pure one-dimension sheaf.
There is an exact sequence 
\[
0 \rightarrow T^{C_1} \rightarrow T \rightarrow T_{C_1} \rightarrow 0.
\]
Since $T^{C_1}$ is a torsion sheaf, $T$ is isomorphic to $T_{C_1}$. 
Thus, $T$ can be obtained from extensions of $\oO_{\Theta_1}(l)$ for some $l>k$. We have $Z(T)<0$.

Suppose that $F$ is a sheaf in $\fF_k$. Since $F$ is a sheaf obtained from extensions of line bundles $\oO_{\Theta_1}(l)$ for $l\leq k$, we have $Z(F)>0$. 
\end{proof}

\begin{prop}
The pair $(Z_{r_1},\aA_k)$ satisfies the HN property, if $r_1$ is a rational number with $k -1< r_1< k$ and $\  z_2,\ldots,z_n \in \bbQ \oplus i \bbQ_{>0}$
In particular, $(Z_{r_1},\aA_k)$ defines a pre-stability condition in the boundary $\partial U(C)$.
\end{prop}
\begin{proof}
It suffices to show two conditions about sequences of monomorphisms and sequences of epimorphisms. 
Suppose, for contradiction, that there is a chain of sub-objects in $\aA_k$
\[
\dots \subset E_{j+1}\subset E_{j} \subset \dots \subset E_1 = E
\]
with $\phi(E_{j+1}) > \phi (E_j) $ for all $j$. There are short exact sequences 
\[
0 \rightarrow E_{j+1} \rightarrow E_j \rightarrow F_j \rightarrow 0
\]
in $\aA$. 
Since $\Imm(Z(A))$ is non-negative for any objects $A \in \aA_k$. 
We have $\Imm Z(E_{j+1}) \leq \Imm Z(E_j)$. 
Since the image of $Z$ is discrete, $\Imm Z(E_j)$ is constant for large enough $j$. Therefore, we can assume $\Imm Z(E_j)$ is constant for all $j$. 
Then, we have $\Imm Z(F_j) =0$, so $\Ree Z(F_j) \leq 0$, which contradicts $\Ree Z(E_j) > \Ree Z(E_{j+1})$.

Next, Suppose we have a chain of quotients in $\aA_k$
\[
E_1 \twoheadrightarrow E_2 \twoheadrightarrow \cdots E_j \twoheadrightarrow E_{j+1}\twoheadrightarrow
\]
with $\phi(E_j) > \phi(E_{j+1})$ for all $j$. 
Suppose this chain does not stabilize.
Taking long exact sequences in cohomology sheaves, we obtain the following chain
\[
\hH^0(E_1) \twoheadrightarrow \hH^0(E_2) \twoheadrightarrow \cdots \hH^0(E_j) \twoheadrightarrow \hH^0(E_{j+1})\twoheadrightarrow.
\]
Since $\Coh(C)$ is Noetherian, it follows for enough large $i$
\[
\hH^{0}(E_i) \cong \hH^{0}(E_{i+1}).
\]
Therefore, we can assume $\hH^{0}(E_i)$ and $\Imm Z(E_i)$ are independent of the choice of $i$. 
There is an exact sequence
\[
0 \rightarrow G_i \rightarrow E_i \rightarrow E_{i+1} \rightarrow 0.
\]
By the assumption, $Z(G_i) \in \bbR_{<0}$. Since the image of $Z$ is discrete, $\lim_{i\to \infty} \Ree Z(E_i) = \infty$. However,  $\hH^{0}(E_i)$ is constant, so $\lim_{i\to \infty} \Ree Z(\hH^{-1}(E_i)) = -\infty$. This contradicts the fact that $Z(F)$ is positive for each $F \in \fF_k$ 
 
\end{proof}

We look stable objects of phase one in $(Z_{r_1},\aA_k)$. We denote the slicing corresponding $(Z_{r_1},\aA_k)$ by $\pP_k$.

\begin{prop}\label{prop:stableofphaseone}
Suppose $\sigma = (Z_{r_1},\aA_k)$ is a stability condition constructed as above.
If $E \in \pP_k(1)$ is stable then $E = \oO_x$ for some $x \in C \backslash C_1$, or $E = \oO_{\Theta_1}(k)[1], \oO_{\Theta_1}(k+1)$.
\end{prop}
\begin{proof}
First, we show $ \oO_{\Theta_1}(k)[1]$ and $\oO_{\Theta_1}(k+1)$ are $\sigma$-stable. 
Suppose that there exist an exact sequence in $\pP_k(1)$ \[
0 \rightarrow K \rightarrow \oO_{\Theta_1}(k)[1] \rightarrow E \rightarrow 0 ,
\]
and $E$ is a non-zero object.
Taking a long exact sequence in cohomology sheaves, we obtain 
\begin{align}\label{exa}
0 \rightarrow \hH^{-1}(K) \rightarrow \oO_{\Theta_1}(k) \rightarrow \hH^{-1}(E) \rightarrow \hH^0(K)\rightarrow 0.
\end{align}
Since $\hH^-1(E) \not \cong \hH^0(K)$, the map $\oO_{\Theta_1}(k) \rightarrow \hH^{-1}(E)$ is non-zero.
If the map $\hH^{-1}(K) \rightarrow \oO_{\Theta_1}(k)$ is non-zero, since $\hH^{-1}(K)$ is not a torsion sheaf, then the rank of its image is one.
However, $\hH^{-1}(E)$ is also pure. Thus, the cohomology sheaf $\hH^{-1}(K) $ is zero.
By the above sequence, we have
\[
Z(\oO_{\Theta_1}(k))\leq Z(E)<0
\]
However, $Z(E)$ is a linear combination of $Z(\oO_{\Theta_1}(l))$ for some $l\le k$, and we conclude that  $Z(E)$ is equal to $Z(\oO_{\Theta_1}(k))$. 
Therefore, $K$ is zero, and $\oO_{\Theta_1}(k)$ is a simple object.
The proof for $\oO_{\Theta_1}(k+1)[1]$ is similar.

Suppose $E$ is stable of phase $\phi=1$ and $E$ is not a skyscraper sheaf on $C \backslash C_1$. Then, $E$ is supported on $C_1$. Since $E$ lies in $\aA_k$, $\hH^{k}(E) = 0 $ for $i \not \in \{-1,0\}$.
If $\hH^{0}(E) \not = 0$, then for some $x \in C_1$ there is a non-zero morphism $E \rightarrow \oO_x$.
For any $x \in C_1$, there is an exact sequence in $\aA_k$
\[
0 \rightarrow \oO_{\Theta_1}(k+1) \rightarrow \oO_x \rightarrow \oO_{\Theta_1}(k)[1]  \rightarrow 0.
\]
Then, we have $\Hom(E,\oO_{\Theta_1}(k+1)) \not =0$ or  $\Hom(E,\oO_{\Theta_1}(k)[1]) \not =0$.
Since all morphisms between simple objects are isomorphism and $\hH^{0}(E) \not = 0$, $E \cong \oO_{\Theta_1}(k+1)$.
It is clear that $E$ is isomorphic to $\oO_{\Theta_1}(k)[1]$ when $\hH^{0}(E) = 0$.

\end{proof}

We will show the following proposition in Proposition \ref{section62}.
\begin{prop}
$(Z_{r_1},\pP_k)$ satisfies the support property.
\end{prop}

For all $i = 1,\ldots,n$ and all rational numbers $r$ with $k-1<r<k$, we construct a stability condition $\sigma_{i,r} = (Z_r,\aA_k)$ such that the central charge $Z_r$ is defined by 
\[
Z(E) = -\chi(E) +z_1\rk_1(E)+\cdots+r\rk_i(E)+\cdots+z_n\rk_n(E),
\] and the abelian category is obtained from tilting with respect to the following torsion pair
\begin{align*}
&\fF_k = \langle\oO_{\Theta_i}(l)\mid l\leq k \rangle_{ex}\\
&\tT_k = {}^{\perp}\fF_k.
\end{align*}

\section{The connected component of $\Stab(C)$}
\subsection{The boundary of the open set $U(C)$}

First, we study the boundary of $U(C)$. We will give a description of codimension one submanifolds of $\Stab(C)$ defining $\partial U(C)$.

\begin{defi}
A stability condition $\sigma$ in $\partial U(C)$ is called general if it is contained by only one submanifold.
\end{defi}

\begin{prop}\label{prop:boundary}
Suppose $\sigma \in \partial U(C)$ is a general point. There is an irreducible component $C_i$ and an integer $k$ such that $\oO_x$ is $\sigma$-stable of the same phase for $x \not \in C_i $ and such that the Jordan-H\"{o}lder filtration of $\oO_x$ for $x \in C_i$ is
\[
0 \rightarrow \oO_{\Theta_i}(k+1) \rightarrow \oO_x \rightarrow \oO_{\Theta_i}(k)[1]  \rightarrow 0.
\]

\end{prop}
\begin{proof}
Each skyscraper sheaf $\oO_x$ is semistable of the same phase for any $\sigma \in \partial U(C)$.
Applying an element of $\widetilde{\GL}^{+}(2,\bbR)$, we can assume the phase of skyscraper sheaves is one. Therefore, we consider only the case when $\sigma \in \partial V(X)$.

Let $\sigma =(Z,\pP)$ be a stability condition in $\partial V(C)$. Since there are stability conditions $\tau \in V(C)$ arbitrarily close to $\sigma$, we have $\Coh(C) \subset \pP([0,1])$. 
If $\Coh(C) \subset \pP((0,1])$, $\sigma$ lies in $V(C)$. 
Therefore, we have $\oO_{\Theta_i}(j) \in \pP(0) $ for some $i = 1,2,\ldots,n$ and $j\in  \bbZ$. 
Replacing $C_i$ by $C_1$, we can assume $\oO_{\Theta_1}(j) \in \pP(0) $ for $j\in  \bbZ$. 
For any $\tau \in V(C)$ and $k\in \bbZ$, we have $\phi_{\tau}(\oO_{\Theta_1}(k)) > \phi_{\tau}(\oO_{\Theta_1}(k-1))$.
Hence, $\phi_{\sigma}(\oO_{\Theta_1}(k)) \geq \phi_{\sigma}(\oO_{\Theta_1}(k-1))$ for any $k \in \bbZ$.
Moreover, if $\oO_{\Theta_1}(j) \in \pP(0)$, we have $\oO_{\Theta_1}(l) \in \pP(0)$ for all integers $l \leq j $. Since $Z(\oO_{\Theta_1}(l)) <0$ for enough large $l$, there is an integer $k$ such that
\[
\oO_{\Theta_1}(k) \in \pP(0), \  \oO_{\Theta_1}(k+1) \in \pP(1). 
\]

We claim that $\pP((0,1])$ can be obtained from tilting with respect to a torsion pair in $\Coh(C)$.
Indeed, if $E$ is a $\sigma$-stable object of phase $\phi$ for some $0<\phi<1$, then $E$ is also $\tau$-stable of phase $\psi$ for $0<\psi<1$ and $\tau \in V(C)$ arbitrarily close to $\sigma$, so $E$ is a coherent sheaf. 
Suppose $E$ is a $\sigma$-stable object of phase one. 
Since $\Coh(C) \subset \pP([0,1])$, we have the cohomology sheaf $\hH^i(E)$ vanishes for $i \not = -1,0$.
We conclude that the pair 
\[
\tT = \Coh(C) \cap \pP((0,1])  \text{ and }  \fF = \Coh(C) \cap \pP((-1,0])
\]
defines a torsion pair in $\Coh(C)$ and the category $\pP((0,1])$ is the corresponding tilt.

Since $\sigma$ is general, the boundary at $\sigma $ is submanifold of $\Stab^{\dagger}(C)$ defined by $Z(\oO_{\Theta_1}(k+1))/Z(\oO_x) \in \bbR_{>0}$. 
Thus, we can assume $Z(F) \not \in \bbR$ for all objects $F$ with $\rk_i(F) \not = 0$ for some $i\in \{2,3,...,n\}$ by moving a little bit along the boundary of $U(C)$. Therefore, we have
\[
\fF = \langle\oO_{\Theta_1}(l)\mid l\leq k \rangle_{ex}.
\]
Moreover, $\pP((0,1])$ is equal to the heart $\aA_k$ of a stability condition constructed in the previous section. By construction, $\oO_{\Theta_1}(k)[1]$ and $\oO_{\Theta_1}(k+1)$ are simple objects in $\aA_k$. 
It follows that there is a Jordan-H\"{o}lder filtration
\[
0 \rightarrow \oO_{\Theta_i}(k+1) \rightarrow \oO_x \rightarrow \oO_{\Theta_i}(k)[1]  \rightarrow 0.
\]
\end{proof}

\begin{defi}
 We call $\sigma$ satisfying assumptions in the above proposition the $(C_{i,k})$-type stability condition for given integers $i,k$.
 We also say that a submanifold $N$ of $\partial U(C)$ corresponds to $(C_{i,k})$ if a general point $\sigma$ in $N$ is $(C_{i,k})$-type.
\end{defi}
We note that the $(C_{i,k})$-type stability condition exists for all $i,k$ by constructing the stability condition $\sigma_{i,r}$ in Section 4.
We consider an autoequivalence $T_{\oO_{\Theta_i}(k)}$ which is obtained by the restriction of the spherical twist along $\oO_{\Theta_i}(k)$ defined in Theorem \ref{thm:twis}.
Since we can compute 
\[
T_{\oO_{\Theta_i}(k)}(\oO_{\Theta_i}(k)) = \oO_{\Theta_i}(k)[-1],\ T_{\oO_{\Theta_i}(k)}(\oO_{\Theta_i}(k+1)) = \oO_{\Theta_i}(k-1), 
\] for all $i = 1,2,\ldots,n$ and all $k$,
we have that a  stability condition $\sigma$ is the $(C_{i,k})$-type if and only if $T_{\oO_{\Theta_i}(k)}(\sigma)$ is the $(C_{i,k-1})$-type.

Recall that for each autoequivalence $\Phi$ and each stability condition $\sigma =(Z,\pP)$, $\Phi(\sigma)$ is defined as follows:
\[
\Phi(\sigma) = (Z\circ \Phi^{-1},\{\Phi(\pP(\phi)) \}).
\]
It defines a left action on $\Stab(C)$ by $\mathrm{Auteq}(C)$.
\begin{prop}
If $\sigma =(Z,\aA) \in \Stab^{\dagger}(C)$ is a $(C_{i,k})$-type stability condition, then $\sigma$ lies in the boundary $\partial U(C)$. Moreover, the central charge $Z$ is defined
\[
Z(E) = -\chi(E) +z_1\rk_1(E)+\cdots+r\rk_i(E)+\cdots+z_n\rk_n(E),
\] for some real number $r$ with $k-1<r<k$ and $z_i \in \bbH$,
and the abelian category is obtained from tilting with respect to the following torsion pair in $\Coh(C)$
\begin{align*}
&\fF_k = \langle\oO_{\Theta_i}(l)\mid l\leq k \rangle_{ex}\\
&\tT_k = {}^{\perp}\fF_k.
\end{align*}

\end{prop}
\begin{proof}
Applying an element of $\widetilde{\GL}^{+}(2,\bbR)$, we can assume each skyscraper sheaf $\oO_x$ is of phase one and $Z(\oO_x) = -1$. 
Thus, by applying the same argument in Proposition \ref{prop:V(C)}, the cohomology sheaves $\hH^i(E)$ vanish unless $i=0$ for all objects $E \in \pP((0,1))$.
Suppose $E$ is stable of phase one.
By a similar argument as above, we conclude that the cohomology sheaves $\hH^i(E)$ vanish unless $i\in\{-1,0\}$. 
Moreover, since skyscraper sheaves on $C\backslash C_i$ are stable of phase one, the $-1$-cohomology sheaf $\hH^{-1}(E)$ is supported on $C_i$.

Since the heart $\pP((0,1])$ is contained in $\langle \Coh(C)[1],\Coh(C)\rangle_{ex}$, the heart $\pP((0,1])$ is obtained by a tilting with respect to a torsion pair $(\tT,\fF)$ on $\Coh(C)$. 
All sheaves in $\tT$ are supported on $C_i$. Thus, the central charge is given by
\[
Z(E) = -\chi(E) +z_1\rk_1(E)+\cdots+r\rk_i(E)+\cdots+z_n\rk_n(E),
\] for real number $r$ with $k-1<r<k$ and $z_i \in \bbH$. 
Since each line bundle $\oO_{\Theta_i}(l)$ for $l\leq k$ lies in $\fF$, we conclude the torsion pair $(\tT,\fF)$ is the same as the torsion pair in the statement.

\end{proof}
We have the following proposition.
\begin{prop}\label{prop:spanedspherical}
Let $\bB$ be a subset of $\mathrm{Auteq}(C)$ generated by the restricted spherical twists of the form $T_{\oO_{\Theta_i}(k)}$ for all $i,k$. 
Then, we have
\[
\Stab^{\dagger}(C) = \bigcup_{\Phi \in \bB}\Phi(\overline{U(C)}).
\]
\end{prop}
\begin{proof}
Let $\sigma$ be a stability condition in $\Stab^{\dagger}(C)$. There is a continuous path $\gamma$ 
\[
\gamma: [0,1] \rightarrow \Stab^{\dagger}(C)
\]
such that $\gamma(0) \in U(C)$ and $\gamma(1) = \sigma$.
Take $t \in [0,1]$ such that $\gamma(t)$ lies in the boundary $\partial U(C)$ and $\gamma((0,t))$ is contained in $U(C)$. Deforming $\gamma$ a little, we can assume that $\gamma(t)$ is a general point in
 the submanifold corresponding to $(C_{i,k})$ for some $i,k$. 
 Then, the boundary of $U(C)$ is defined locally by $Z(\oO_{\Theta_1}(k+1))/Z(\oO_x) \in \bbR_{>0}$ and $U(C)$ is the side where $\Imm Z(\oO_{\Theta_1}(k+1))/Z(\oO_x) < 0$. 
 Applying the restricted spherical twist $\Phi = T_{\oO_{\Theta_i}(k)}$, $\Phi(\gamma(t))$ is in the submanifold corresponding to $(C_{i,k-1})$ and $\Phi(\gamma(t,t+\varepsilon))$ is contained in $U(C)$ for some positive number $\varepsilon$. 
 Indeed, if $\Imm Z(\oO_{\Theta_1}(k+1))/Z(\oO_x) < 0$, then we have $\Imm Z(\Phi^{-1}\oO_{\Theta_1}(k+1))/Z(\Phi^{-1}\oO_x) > 0$. Thus, $\Phi$ acts on $G(C)$ via a reflection. 
 
 Consider the wall and chamber structure on $\Stab^{\dagger}(C)$ with respect to $[\oO_x]$ defined in \cite[Section 9]{Bri:08}. 
 For any compact subset $\mM$ there is a finite collection $\{\wW_{i}\mid i \in I\}$ of real codimension one submanifolds of $\Stab^{\dagger}(C)$ such that any connected components $\cC \in \mM\backslash \cup \wW_{i}$ satisfies the following property:
 If a object $[E]$ with a class $[E] = [\oO_x]$ is semistable (resp. stable) in $\sigma$ for some $\sigma\in \cC$, then $E$ is semistable (resp. stable) in $\sigma$ for all $\sigma \in \cC$.
 
 Since $\gamma([0,1])$ is compact, then there are finite walls intersecting the path  $\gamma([0,1])$. Repeating the above argument, we conclude that there is a autoequivalence $\Phi \in \bB$ such that $\sigma \in \Phi(\overline{U(C)})$.

\end{proof}

\subsection{The covering property of $\Stab^{\dagger}(C)$}
Next, we study the forgetful map $\pi: \Stab^{\dagger}(C) \rightarrow G(C)^{\vee}\otimes \bbC$. 
We show $\pi$ is a covering map on an open set $\pP^{+}_0(C) \subset G(C)^{\vee}\otimes \bbC$.
First, we consider a pairing on the Grothendieck group $G(C)$. 
Let $S \rightarrow B$ be a relatively minimal elliptic surface that has a fiber $S_0$ isomorphic to the reducible Kodaira curve $C$. We denote by $i$ a closed immersion $C \hookrightarrow S$.

\begin{lem-defi}\label{lem:squre}
Let $\langle -,- \rangle$ be a pairing on $G(C)$ defined by
\[
\langle [E], [F] \rangle = -\chi(i_*E,i_*F) = \sum_{j} (-1)^{j+1}\dim \Hom_S(i_*E,i_*F[j]).
\]
Then, the pairing $\langle -,- \rangle$ is a negative semi-definite symmetric form.
\end{lem-defi}
\begin{proof}
Since the canonical bundle $\omega_S$ on $S$ is trivial on each fiber, it follows from Serre duality that the pairing is symmetric.

By the definition of the pairing, we compute that $\langle [\oO_{\Theta_i}(-1)],[\oO_{\Theta_i}(-1)]\rangle = -2$ for any $i$, and $\langle [\oO_{\Theta_i}(-1)],[\oO_{\Theta_j}(-1)]\rangle$ is equal to the total number of the intersection points with multiplicities for any $i$ and $j$ with $i \not=j$.
The intersection matrix $A = (\langle [\oO_{\Theta_i}(-1)],[\oO_{\Theta_j}(-1)]\rangle)_{i,j}$ is equal to the Cartan matrix of the affine Dynkin diagram corresponding to $C$.
Since this Cartan matrix is negative semi-definite, so is the form $\langle -,- \rangle$.

\end{proof}

It also shows that the classes $[\oO_x] $ and $a_1[\oO_{\Theta_1}(-1)]+\cdots +a_n[\oO_{\Theta_n}(-1)]$ are numerical trivial with respect to the pairing $\langle -,- \rangle$ where $a_i$ is a multiplicity of the irreducible component $C_i$.
There is vector subspace $V_0\subset G(C)\otimes \bbR$ such that the pairing is zero on $V_0$.

\begin{prop}
Let $\sigma$ be a stability condition in $\Stab^{\dagger}(C)$ and $E$ be a $\sigma$-stable object.
Then, 
\[
\langle E,E \rangle \in \{-2,0\}.
\]
\end{prop}
\begin{proof}
Since the pairing is negative semi-definite, and a lattice $(G(C),\langle-,-\rangle)$ is even, it is enough to show that $\langle E,E \rangle\geq-2$ for any $\sigma$-stable object $E$. Consider the closed immersion $i:C \rightarrow S$, we have a distinguished triangle
\[
E[1] \rightarrow i^*i_*E  \rightarrow E.
\]
Applying $\Hom(-,E)$ yields a long exact sequence
\[
\Hom(E[2],E) \rightarrow \Hom(E,E) \rightarrow \Hom(i^*i_*E,E) \rightarrow \Hom(E[1],E).
\]
Since $E$ lies in the heart of a t-structure, we have $\Hom(E[1],E)=0$ and
\[
 \Hom(i_*E,i_*E)\cong\Hom(E,E) \cong \bbC.
\]

Therefore, 
\begin{align*}
\chi(i_*E,i_*E) \geq -\Hom(i_*E,i_*E)-\Hom(i_*E,i_*E[2]) \geq -2.
\end{align*}

\end{proof} 
Let $\Delta(C)$ be the subset in $G(C)$ defined by
\[
\Delta(C) = \{ \delta \in G(C)\mid \langle \delta, \delta \rangle =-2 \}.
\]
We consider an open set $\pP_0(C)$ in $G(C)^{\vee}\otimes \bbC$ defined by
\[
\pP_0(C) =\left\{ Z\in G(C)^{\vee}\otimes \bbC \;\middle|\;
\begin{array}{l}
Z(\delta) \not = 0 \text{ for all }\delta \in \Delta(C)\\
\Ree (Z\circ s) \text{ and } \Imm(Z\circ s) \text{ are linearly independent in } V_0^{\vee}
\end{array}
\right\}.
\]
Here $s$ denotes the inclusion $V_0 \hookrightarrow G(C)\otimes \bbC$.

\begin{prop}\label{prop:cov}
The restriction
\[
\pi: \pi^{-1}(\pP_0(C)) \rightarrow \pP_0(C)
\]
is a covering map.
\end{prop}
\begin{proof}
Let $\sigma =(Z,\aA)$ be a stability condition with $Z\in \pP_0(C)$. Since $\Ree(Z \circ s)$ and $\Imm(Z\circ s)$ are linearly independent, $Z$ induces an isomorphism between $V_0$ and $\bbC$ as $\bbR$-vector spaces. 
Therefore, for any non-zero vector $v \in G(C)\otimes \bbC$, if $Z(v)$ is zero then $\langle v,v\rangle <0$.

First, we claim that
\[
M = \inf \{|Z(\delta)|\mid \delta \in \Delta(C)\} >0.
\]
Consider a basis $\{[\oO_x],a_1[\oO_{\Theta_1}(-1)]+\cdots + a_n[\oO_{\Theta_n}(-1)],[\oO_{\Theta_2}(-1)],\ldots,[\oO_{\Theta_n}(-1)]\}$ on $G(C)\otimes \bbC$ where $a_i$ is a multiplicity of $C_i$.
For $\delta \in \Delta(C)$, we will denote by $\delta_i$ the $i$-th component of the class $\delta$ with respect to the basis.
Replacing $C_1$ by other component, we assume the dual graph of $C_2\cup C_3\cup\cdots C_n$ is a Dynkin diagram $\Gamma$ corresponding affine Dynkin diagram of $C$. 
Denote a subspace $V_{-}$ generated by $[\oO_{\Theta_2}(-1)],\ldots,[\oO_{\Theta_n}(-1)]$, and the norm $(-,-)$ which is induced by the restriction of $-\langle-,-\rangle$.
The basis $\{[\oO_{\Theta_2}(-1)],\ldots,[\oO_{\Theta_n}(-1)]\}$ defines a root system on $V_{-}$ corresponding $\Gamma$.
By the construction of root systems of the Dynkin diagrams \cite[12.1]{hum:12}, the set $\{\alpha\mid (\alpha,\alpha) = 2\}$ is the root system defined by the basis.
Therefore, when we fix integers $\delta_0,\delta_1$, there are only finitely many integral classes $\delta$ lying $\Delta(C)$. 
We conclude there are also finitely many integral classes $\delta_0,\delta_1$ for given $\delta_2,...,\delta_n$, if $|Z(v)|<1$. Since $|Z(\delta)| \not = 0$, the claim follows.

Now define a quadratic form
\[
Q(v) := \langle v,v\rangle + \frac{2}{M^2}|Z(v)|^2.
\]
Since $\langle - ,- \rangle$ is negative definite on $\Ker Z$, $Q$ is also negative definite on $\Ker Z$. And, $Q(v) \geq 0 $ for all classes $v$ with $\langle v ,v \rangle =0,-2$.
Let $\pP_Z(Q)$ be a connected component containing $Z$ of the set $\pP(Q)$ in $G(C)^{\vee}\otimes \bbC$ defined by
\[
\pP(Q) = \{Z \in \pP_0(C)\mid Q \text{ is negative definite on }\Ker Z\}.
\]
Then, for any stability conditions $\sigma$ whose central charge $Z$ lies in $\pP_Z(Q)$, $\sigma$ satisfies the support property with respect to $Q$.

Consider an open neighborhood $U = \pP_Z(Q) \subset \pP_0(C)$ of $Z\in \pP_0(C)$. For all stability conditions $\sigma \in \Stab^{\dagger}(C)$ with $\pi(\sigma) = Z$, a connected component $\widetilde{U}$ containing $\sigma$ of $\pi^{-1}(U)$ is isomorphic to $U$. 
Suppose $\tau$ is a stability condition in $\Stab(C)$ with $\pi(\tau) = Z^{\prime} \in U$. 
It follows from Theorem \ref{thm:deformation} that there is a connected component $\widetilde{U}$ containing $\tau$ of $\pi^{-1}(U)$ such that the restriction $\pi : \widetilde{U} \rightarrow U$ is a covering map. 
In particular, since $Z$ lies in $\pP_Z(Q)$, any connected components contain a stability condition $\sigma$ with $\pi(\sigma) =Z$, which proves that the forgetful map $\pi: \pi^{-1}(\pP_0(C)) \rightarrow \pP_0(C)$ is covering map on the target.
\end{proof}

\begin{lem}
Let $\pP_0^{+}(C)$ be a connected component containing $Z = -\chi + \sqrt{-1}(\rk_1+\cdots +\rk_n)$ of $\pP_0(C)$.
Then, 
\[
\pi(\Stab^{\dagger}(C)) = \pP_0^{+}(C).
\]
\end{lem}
\begin{proof}
Since $\pi$ is a covering map on $\pP_0^{+}(C)$, $\pi(\Stab^{\dagger}(C))$ contains $\pP_0^{+}(C)$. Then, it is enough to show the reverse inclusion.
First, we claim that $\pi(\overline{V(C)}) \subset \pP_0^{+}(C)$. By the construction of stability conditions in $V(C)$, we have $\pi(V(C)) \subset  \pP_0^{+}(C)$, so we show that the boundary $\partial {V(C)}$ is contained in $\pP_0^{+}(C)$.

Let $\sigma = (Z,\aA)$ be a stability condition in $\partial {V(C)}$ with $Z = -\chi +z_1\rk_1+\cdots +z_n\rk_n$. Suppose $\Ree(Z\circ s)$ and $\Imm(Z\circ s)$ are not linearly independent in $V_0^{\vee}$. Since we have $\Imm(z_1+\cdots +z_n) =0$, all coefficients $z_1,\ldots,z_n$ lies in the real line, and all objects in $\dD^b(C)$ are semistable.
There is a sequence in the corresponding slicing $\pP(0)$
\[
\cdots \hookrightarrow \oO_C(-nP)\hookrightarrow \oO_C(-(n-1)P) \hookrightarrow \cdots \hookrightarrow\oO_C(-(k+1)P) \hookrightarrow \oO_C(-kP)
\] for some point $P\in C$ and some integer $k$ with $Z(\oO_C(-kP))>0$.
 This contradicts the fact that $\pP(0)$ is of finite length and the support condition, so $\Ree(Z\circ s)$ and $\Imm(Z\circ s)$ are linearly independent.

Let $\delta \in \Delta(C)$ and $\sigma =(Z,\pP)\in \Stab^{\dagger}(C)$. Suppose $Z(\delta)=0$ to obtain a contradiction. 
Let $\delta_i$ denote the $i$-th component of the vector $\delta \in G(C)\otimes\bbC$ with respect to $[\oO_x],[\oO_{\Theta_1}],...,[\oO_{\Theta_n}]$. 
Let $J$ be an index $j$ where $\delta_j$ is non-zero. 
By the assumption, coefficients $z_j$ of the central charge $Z$ lie in the real line $\bbR$ for all $j\in J$,and $\sigma \in \partial V(C)$. 
It follows from the above argument that some coefficients $z_i$ lie in the upper half plane.

The boundary of $V(C)$ in an open neighborhood of $\sigma$ consists of an union of codimension one submanifolds corresponding $(C_{j,k_j})$ for all $j$ and some integers $k_j$. 
Then the line bundle $\oO_{\Theta_j}(k_j+1)$ and the shifted line bundle $\oO_{\Theta_j}(k_j)[1]$ on $C_j$ are $\sigma$-semistable objects of phase one. 
Skyscraper sheaves $\oO_x$ are also $\sigma$-semistable.
 Applying the same argument in Theorem \ref{prop:boundary}, $\pP((0,1])$ is obtained from tilting with respect to a torsion pair $(\Coh(C)\cap\pP((0,1]),\Coh(C)\cap\pP((-1,0]))$. 
 All coherent sheaves on $\bigcup_{j\in J}C_j$ are $\sigma$-semistable of phase zero or one.
 By replacing $\delta$ by $-\delta$, all signs of non-zero components $\delta_1,\ldots,\delta_n$ are positive.
 Since there is a coherent sheaf $E$ with $[E]=\delta$, then $Z(\delta) \not = 0$.
 This contradicts our assumption. 
 Thus, we have
 \[
 \overline{V(C)}\subset \pP_0^{+}(C).
 \]

By applying the action of $\widetilde{\GL}^{+}(2,\bbR)$ and $\bB$, we have $\pi(\Stab^{\dagger}(C)) \subset \pP_0^{+}(C)$.
\end{proof}

Let $\mathrm{Auteq}^0(C)$ be the kernel of $\mathrm{Auteq}(C) \rightarrow \mathrm{Aut}(G(C))$ and $\mathrm{Auteq}^0_*(C)$ be the subgroup of $\mathrm{Auteq}^0(C)$ preserving the connected component $\Stab^{\dagger}(C)$. 
Contrary to the case of K3 surfaces, there are non-trivial autoequivalences that act trivially on $\Stab^{\dagger}(C)$

\begin{lem}

Let 
\begin{align*}
&\mathrm{Aut}_{\mathrm{tri}}(C) =\{f \in \mathrm{Aut}(C) \mid f_* \mathrm{\ is\ trivial\ in\ }G(C)\},\\
&\mathrm{Pic}^0_{\mathrm{tri}}(C) =\{\lL \in \mathrm{Pic}^0(C)\mid \lL|_{\Theta_i} \cong \oO_{\Theta_i} \mathrm{\ for\ all\ reduced\ curves\ }\Theta_i\}.
\end{align*}
Then, all autoequivalences $\Phi \in \mathrm{Pic}^0_{\mathrm{tri}}(C) \rtimes\mathrm{Aut}_{\mathrm{tri}}(C)$ induce the trivial action on $\Stab^{\dagger}(C)$. 
\end{lem}
\begin{proof}
Let $f \in \mathrm{Aut}_{\mathrm{tri}}(C)$. Since $f_*$ sends coherent sheaves to coherent sheaves,  the action on $\overline{U(C)}$ is trivial. 
Take any stability condition $\tau \in \Stab^{\dagger}(C)$. There is a path $\gamma\colon[0,1] \rightarrow \Stab^{\dagger}(C)$ such that $\gamma(0) \in U(C)$ and $\gamma(1) = \tau$.
Consider $\sup\{t\in[0,1] \mid f_*\gamma(t) = \gamma(t)\}$.
Since $\pi$ is a local homeomorphism and $f$ acts trivially on $G(C)$, we have $\sup\{t\in[0,1] \mid f_*\gamma(t) = \gamma(t)\} = 1$. Therefore,  the action $f_*$ is trivial on $\Stab^{\dagger}(C)$.

Since for any line bundle $\lL \in \mathrm{Pic}^0_{\mathrm{tri}}(C) $ the functor $-\otimes \lL$
acts trivially on $G(C)$, its action on $\overline{U(C)}$ is trivial. 
Thus, applying a similar argument, we can see that the action $\otimes \lL$ is trivial on $\Stab^{\dagger}(C)$.
\end{proof}

\begin{thm}\label{thm:main1}
The map $\pi :\Stab^{\dagger}(C) \rightarrow \pP_0^{+}(C)$ is a covering map.
Moreover, the quotient group $\mathrm{Auteq}^0_*(C) /(\mathrm{Pic}^0_{\mathrm{tri}}(C) \rtimes\mathrm{Aut}_{\mathrm{tri}}(C))$ is isomorphic to the deck transformation group.
\end{thm}

\begin{proof}
We show that the map
\[
\mathrm{Auteq}^0_*(C) \rightarrow \mathrm{Deck}(\pi)
\]
is surjective, and the kernel is isomorphic to $\mathrm{Pic}^0_{\mathrm{tri}}(C) \rtimes\mathrm{Aut}_{\mathrm{tri}}(C)$. Here, $\mathrm{Deck}(\pi)$ is the deck transformation group of $\pi$.

First, we show the surjectivity. Suppose $\sigma$ and $\tau$ are stability conditions with the same central charge. Since $\pi$ is a covering map, we can assume $\sigma =(Z,\Coh(C)) \in V(C)$ with $Z(E) = -\chi(E) + z_1\rk_1(E) + \cdots +z_n \rk_n(E)$.

By Theorem \ref{prop:spanedspherical}, there is an automorphism $\Phi \in \bB$ such that $\Phi(\tau) = (Z_{\Phi},\aA_{\Phi})$ lies in $\overline{U(C)}$. 
Moving $\sigma$ a bit in $V(C)$, we can also assume $\Phi(\tau) \in U(C)$. Since $\Phi\in \bB$ preserves classes $[\oO_x]$ and $\alpha=a_1[\oO_{\Theta_1}(-1)]+\cdots +a_n[\oO_{\Theta_n}(-1)]$, $[\oO_x]$ is $\Phi(\tau)$-stable of phase one and $Z(\oO_x)=-1$ by composing even shifts. 
Here, $a_i$ is a multiplicity of curve $C_i$.
Thus, we conclude that  $\Phi(\tau) \in V(C)$ and $\aA_{\Phi} = \Coh(C)$. 

Let $L$ be a sublattice of $G(C)$ generated by $[\oO_{\Theta_1}(-1)],\ldots,[\oO_{\Theta_n}(-1)]$. 
We can compute 
\[
\rR:=\{\delta \in L\mid \langle \delta,\delta \rangle=-2 \} = \{\delta + m\alpha\mid \delta \in \rR_0,\ m \in \bbZ\}
\]where $\rR_0$ is a subset of $L$ consisting a class $\delta$ with $\langle \delta,\delta \rangle=-2$ and $\lvert\rk_i\delta\rvert < \rk_i[\oO_{C_i}(-1)]$ for some $i$.
Define $L^+ = \bigoplus_{i}\bbZ_{\geq0}[\oO_{\Theta_i}(-1)]$.
Take $\beta_i\in\rR_0 $ and $m_i \in \bbZ$ such that $\Phi^{-1}([\oO_{\Theta_i}(-1)]) = \beta_i + m_i \alpha$. By replacing $\beta_i$ by $\beta_i+\alpha$, we assume $\beta_i\in L^+$.
Since $\Imm Z(\Phi^{-1}([\oO_{\Theta_i}(-1)]))  >0$, we have $m_i\geq0$.
The functor $\Phi$ preserves the class $\alpha$, and so $m_i =0$ for all $i$.
Indeed,we have
\[
\alpha = \Phi(\alpha) = \sum_{i} \beta_i + \sum m_i \alpha.
\]
We conclude $m_i =0$.
Let $L_{p}^{+} \subset L^+$ be a subset generated by $[\oO_{\Theta_i}(-1)]$ where the multiplicity $a_i$ of the corresponding curve $C_i$ is the integer $p$.
$\Phi$ preserves $L_{p}^{+}$   and the class $\sum_{[\oO_{\Theta_i}(-1)]\in L_{p}^{+}}[\oO_{\Theta_i}(-1)]$ for any integer $p$
Thus, $\Phi([\oO_{\Theta_i}(-1)]) = [\oO_{\Theta_{n_i}}(-1)]$ for some $n_i$.

Since $\Phi$ preserves the form $\langle-,-\rangle$, $\Phi$ is obtained from a graph automorphism.
By composing the corresponding automorphism $C\rightarrow C$, we can assume $\Phi$ acts trivially on $G(C)$. Therefore, we have $\Phi(\tau) =\sigma$.

Next, we compute the kernel of the surjective map $\mathrm{Auteq}^0_*(C) \rightarrow \mathrm{Deck}(\pi)$.
Suppose $\Phi \in \mathrm{Auteq}^0_*(C)$ is trivial on fibers of $\pi$. Consider $\sigma \in V(C)$. 
We have $\Phi(\sigma)= \sigma$. Since $\Phi(\oO_x)$ is $\sigma$-stable coherent sheaf of phase 1, $\Phi(\oO_x)$ is also a skyscraper sheaf, and $\Phi$ takes skyscrapers to skyscrapers. 
By \cite[Cor 5.23]{Huy:06}, there exists a line bundle $L\in \mathrm{Pic}(C)$ and $f \in \mathrm{Aut}(C)$ such that $\Phi(-) = f_*(-\otimes L) $. 
Since $f_*(-\otimes L)$ acts trivially on $G(C)$, we have $L\in \mathrm{Pic}^0_{\mathrm{tri}}(C)$ and $f \in \mathrm{Aut}_{\mathrm{tri}}(C)$.
\end{proof}

\section{Proofs of support properties}

In this section, we show that stability conditions constructed in Section 4 satisfy the support properties.
To show this statement, we use the following result.
\begin{lem}[{\cite[Lemma A.6.]{BMS:16}}]
Let $Q$ be a quadratic form on $\Lambda\otimes \bbR$. Assume $\sigma$ is a pre-stability condition with respect to $\Lambda$ such that the kernel of $Z$ is negative semi-definite with respect to $Q$. If $E$ is strictly $\sigma$-semistable with Jordan-H\"{o}lder factors $E_1,\ldots, E_m$ and if $Q(E_i)\geq0$ for all $i = 1,\ldots,m$, then $Q(E)\geq0$.
\end{lem}

Let $\sigma = (Z,\pP)$ be a stability condition constructed in Section 4.
 Since the central charge $Z$ lies in $\pP_0^+(C)$, by the same argument in the proof of Proposition \ref{prop:cov}, there is a quadratic form $Q$ such that the kernel of $Z$ is negative definite with respect to $Q$ and $Q(E)\geq0$ for any $\sigma$-stable objects $E$.
Thus, it is enough to show $\pP(\phi)$ is of finite length for any $\phi$.

\begin{prop} \label{section61}
Let $\sigma = (Z,\Coh(C))$ be a stability condition whose central charge is defined by
\[
Z(E) = -\chi(E) + z_1\rk_1(E) + \cdots + z_n\rk_n(E).
\]
Then, the corresponding addictive subcategories $\pP(\phi)$ are of finite length.
\end{prop}
\begin{proof}
Suppose $E$ is a $\sigma$-semistable sheaf of phase $\phi$ for $\phi \in(0,1)$.
Take an exact sequence in $\pP(\phi)$
\[
0 \rightarrow F \rightarrow E \rightarrow G \rightarrow 0.
\]
If $\rk_i(E) = \rk_i(F)$ for all $i =1,\ldots,n$, then the phase of $G$ is not $\phi$.
Then, $\pP(\phi)$ is of finite length for $\phi \in(0,1)$.
When $\phi$ is one, all sheaves in $\pP(1)$ are torsion sheaves. Then, $\pP(1)$ is also of finite length.
\end{proof}

\begin{prop}\label{section62}
Let $\sigma = (Z,\aA_k)$ be a stability condition whose central charge is defined by
\[
Z(E) = -\chi(E) + r_1\rk_1(E) + \cdots + z_n\rk_n(E).
\]where $r$ is a rational number with $k-1<r<k$.
Then, the corresponding addictive subcategories $\pP(\phi)$ are of finite length.
\end{prop}
\begin{proof}
Recall that the abelian category $\aA_k$ is defined by the following torsion pair:
\begin{align*}
&\fF_k = \langle\oO_{\Theta_1}(l) \mid l\leq k \rangle_{ex}\\
&\tT_k = {}^{\perp}\fF_k
\end{align*}
A Similar argument in the above proposition apply the case of $0<\phi<1$, we can see that $\pP(\phi)$ is of finite length for $\phi \in(0,1)$.
Recall that simple object $\pP(1)$ is $\oO_{\Theta_1}(k)[1]$ or $\oO_{\Theta_1}(k+1)$ by Proposition \ref{prop:stableofphaseone}.
When $\phi$ is one, we consider a chain of sub-objects in $\pP(1)$
\[
\dots \subset E_{j+1}\subset E_{j} \subset \dots \subset E_1 = E
\] whose factors are simple objects $\oO_{\Theta_1}(k)[1]$ or $\oO_{\Theta_1}(k+1)$ .

When the cokernel  $E_{j}/E_{j+1}$ is isomorphic to $\oO_{\Theta_1}(k)[1]$, we have a long exact sequence 
\[
0 \rightarrow \hH^{-1}(E_{j+1}) \rightarrow \hH^{-1}(E_{j}) \rightarrow \oO_{\Theta_1}(k)\rightarrow \cdots.
\] 
Since the cokernel of the map $\hH^{-1}(E_{j+1}) \rightarrow \hH^{-1}(E_{j})$ is contained by $\oO_{\Theta_1}(k)$, we conclude that $\hH^{-1}(E_{j})/\hH^{-1}(E_{j+1})$ is isomorphic to $0$ or a line bundle on $\Theta_1$.
Thus we conclude that $\rk_1(\hH^{-1}(E_{j+1})) < \rk_1(\hH^{-1}(E_{j}))$ or $\hH^{-1}(E_{j+1}) \cong \hH^{-1}(E_{j})$.
If $\hH^{-1}(E_{j+1}) \cong \hH^{-1}(E_{j})$, then there is an exact sequence
\[
0\rightarrow\oO_{\Theta_1}(k)\rightarrow \hH^0(E_{j+1}) \rightarrow \hH^0(E_j) \rightarrow 0
\] 
and  $Z(\hH^0(E_{j+1})) > Z(\hH^0(E_j))$

When the cokernel  $E_{j}/E_{j+1}$ is isomorphic to $\oO_{\Theta_1}(k+1)$, we have $\rk_1(\hH^{0}(E_{j+1})) < \rk_1(\hH^{0}(E_{j}))$ and $\hH^{-1}(E_{j+1}) \cong \hH^{-1}(E_{j})$. We assume $\hH^{-1}(E_{n})$ dose not depend on the choice of $n$.  In this case when $\hH^{-1}(E_{j+1}) \cong \hH^{-1}(E_{j})$, we also have $Z(\hH^0(E_{j+1})) > Z(\hH^0(E_j))$.

To show that this sequence stabilizes, we consider a rational number $Z(\hH^0(E_{j}))$.
By the definition of the central charge, $Z(\hH^0(E_{j}))$ is negative and the image of $Z$ is discrete.
In the both case, $Z(\hH^0(E_{j}))$ is strictly decreasing for all $j$, a contradiction.
Therefore, this sequence stabilizes.

\end{proof}

\input{refs.bbl}

\end{document}

%% file: refs.bbl
\newcommand{\etalchar}[1]{$^{#1}$}
\providecommand{\bysame}{\leavevmode\hbox to3em{\hrulefill}\thinspace}
\providecommand{\MR}{\relax\ifhmode\unskip\space\fi MR }
\providecommand{\MRhref}[2]{%
  \href{http://www.ams.org/mathscinet-getitem?mr=#1}{#2}
}
\providecommand{\href}[2]{#2}